\documentclass[reqno]{amsart}
\usepackage[english]{babel}
\usepackage{amscd,amssymb,amsmath,amsfonts,latexsym,amsthm}
\usepackage{inputenc}
\usepackage{graphicx}
\newtheorem{thm}{Theorem}[section]
\newtheorem{cor}[thm]{Corollary}

\newtheorem{prop}[thm]{Proposition}
\newtheorem{defn}[thm]{Definition}

\newtheorem{rem}[thm]{Remark}
\numberwithin{equation}{section}

\DeclareMathOperator{\Hom}{Hom}
\DeclareMathOperator{\Orb}{Orb}
\begin{document}


\title[Infinitesimal deformations of naturally graded filiform Leibniz algebras]
{Infinitesimal deformations of naturally graded filiform Leibniz algebras}%
\author{Khudoyberdiyev A.Kh. and Omirov B.A.}

\address{[A.\ Kh.\ Khudoyberdiyev and B.\ A.\ Omirov] Institute of Mathematics, National University of Uzbekistan,
Tashkent, 100125, Uzbekistan.} \email{khabror@mail.ru, omirovb@mail.ru}\

%

\begin{abstract}
We describe infinitesimal
deformations of complex naturally graded filiform Leibniz
algebras. It is known that any $n$-dimensional filiform Lie algebra can
be obtained by a linear integrable deformation of the naturally
graded algebra $F_n^3(0)$. We establish that in
the same way any $n$-dimensional filiform Leibniz algebra can be
obtained by an infinitesimal deformation of the filiform Leibniz
algebras $F_{n}^1,$ $F_{n}^2$ and $F_{n}^3(\alpha)$. Moreover, we
describe the linear integrable deformations of above-mentioned algebras with
a fixed basis of $HL^2$ in the set of all $n$-dimensional
Leibniz algebras. Among these deformations we found one new rigid
algebra.

\end{abstract}

\maketitle \textbf{Mathematics Subject Classification 2010}:
17A32, 17A70, 17B30, 13D10, 16S80.

\textbf{Key Words and Phrases}: Leibniz algebra, the second group of
cohomology, naturally gradation, filiform algebra, infinitesimal deformation, linear integrable deformation,
rigid algebra.

\section{Introduction.}

Deformations of different algebraic and analytic objects are important aspect
if one studies their properties. They characterize the local behavior in a small
neighborhood in the variety of a given type objects.
A geometric picture of the deformations is obtained by considering the variety $M$ of all
those bilinear maps ("products") of the underlying vector space into itself which satisfy the conditions defining of the variety. An algebra structure of the variety represents a point $m$ of $M$, and deformations of the structure are represented by points of $M$ near $m$. Thus the study of deformations of these algebras is a special case of a study of local geometric properties of varieties.

Classical deformation theory of associative and Lie algebras began with
the works of Gerstenhaber \cite{Gersten} and Nijenhuis-Richardson \cite{Nijen} in the
1960s. They studied one-parameter deformations and established the
connection between Lie algebra cohomology and infinitesimal deformations. After these works formal
deformation theory was generalized in different categories. In fact, in the last fifty years, deformation theory has played an important role in algebraic geometry. The main goal is the classification of families of geometric objects
when the classifying space (the so called moduli space) is a reasonable geometric space. In particular, each point of our moduli space corresponds to one geometric object (class of isomorphism).
The theory of deformations is one of the effective approach in investigating of solvable and nilpotent Lie algebras (see for example, \cite{Fial2,Fial3,Khak,Mil}).

In this paper we study infinitesimal deformations of some nilpotent Leibniz algebras.
Recall, that Leibniz algebras are generalization of Lie algebras \cite{Lo 2,Lo 3} and it is natural to apply the theory of deformations to the study of these algebras. Particularly, the problems which were studied in \cite{Fial2,Khak,Mil} and others can be considered from point of Leibniz algebras view.

From algebraic geometry it is known that an algebraic variety is a union of irreducible components. The closures of orbits of rigid algebras give irreducible components of the variety. That is why the finding of rigid algebras is crucial problem from the geometrical point of view.

Due to \cite{Balavoine} we can apply the general principles for deformations and rigidity of Leibniz algebras. Namely, it is proved that nullity of the second cohomology group ($HL^2(L, L) = 0)$ gives a sufficient condition for rigidity. In addition, it is established that Leibniz algebras for which every formal deformation is equivalent to a trivial deformation are rigid.

One of the inherent properties of finite-dimensional Leibniz
algebras consist of the existence of nilpotent single-generated
Leibniz algebras (so-called null-filiform algebras), which are
Leibniz algebras of maximal nilindex. It is known that in each
dimension all of those algebras are isomorphic to the algebra $NF_n$
\cite{Ayup} and this algebra is rigid in the variety of
$n$-dimensional nilpotent Leibniz algebras. In \cite{KhudOmir} infinitesimal 
deformations of the algebra $NF_n$ are studied. It was proved that any single-generated Leibniz
algebra (which is solvable) is a linear integrable deformation of
$NF_n$. Moreover, it is shown that the closure of the set of all single generated Leibniz
algebras forms an irreducible component of the variety of
$n$-dimensional Leibniz algebras (denoted by $Leib_n$).

Firstly the notion of filiform algebra was introduced by M.Vergne in \cite{Ver} as an algebra of maximal nilindex in the variety of Lie algebras. Namely, naturally graded filiform Lie algebras are classified and it is proved that any filiform Lie algebra is represented by linear integrable deformation of special filiform Lie algebra.

In \cite{Ayup} for Leibniz algebras by the approach of M.Vergne
similar description was obtained. In particular, up
to isomorphism there are only three naturally graded filiform
Leibniz algebras ($F_{n}^1,$ $F_{n}^2$ and $F_{n}^3(\alpha),
\alpha\in\{0;1\}$).

The structure of the paper is as follows: In Section 2 we give the
necessary definitions and facts. Section 3 is divided into three subsections:
Subsection 3.1 deals with the second group Leibniz cohomology of
the algebra $F_{n}^1$ and the description of some linear
integrable deformations of $F_{n}^1$ in the variety of $Leib_n$.
Among these deformations we indicate unknown till now rigid
Leibniz algebra. In Subsection 3.2 we describe infinitesimal
deformations of the algebra $F_{n}^2$ and its linear integrable
deformations with respect to chosen basis of
$HL^2(F_{n}^2,F_{n}^2)$. In Subsection 3.3 for the algebra
$F_{n}^3(0)$ we establish that Lie infinitesimal deformations
together with three indicated Leibniz infinitesimal deformations
form the space of all Leibniz infinitesimal deformations.

Throughout the paper we consider finite-dimensional vector spaces and algebras over the field
of complex numbers. Moreover, in the multiplication table of a Leibniz algebra the omitted
products and in the expansion of 2-cocycles the omitted values are assumed to be zero.

\section{Preliminaries.}

In this section we give necessary definitions and known results.

\begin{defn}
A Leibniz algebra over $F$ is a vector space $L$
equipped with a bilinear map, called bracket,
$$[-,-]:  L \times  L \rightarrow  L $$
satisfying the Leibniz identity:
$$[x,[y,z]]=[[x,y],z]-[[x,z],y],$$
for all $x,y,z \in  L.$
\end{defn}

The set $Ann_r(L)=\{x \in L: [y,x]=0, \ \forall y \in L\}$ is
called \emph{the right annihilator of a Leibniz algebra $L$}.
Note that $Ann_r(L)$ is an ideal of $L$ and for any
$x, y \in L$ the elements $[x,x]$, $[x,y]+ [y,x]\in Ann_r(L).$

We call a vector space $M$ a module
over $L$ if there are two bilinear maps:
$$[-,-]:L\times M \rightarrow M \qquad \text{and} \qquad [-,-]:M\times L \rightarrow M$$
satisfying the following three axioms
\begin{align*}
[m,[x,y]] & =[[m,x],y]-[[m,y],x],\\
[x,[m,y]] & =[[x,m],y]-[[x,y],m],\\
[x,[y,m]] & =[[x,y],m]-[[x,m],y],
\end{align*}
for any $m\in M$, $x, y \in L$.

Given a Leibniz algebra $L$, let $C^n(L,M)$ be the space of all
$F$-linear homogeneous mappings $L^{\otimes n} \rightarrow M$, $n
\geq 0$ and $C^0(L,M) = M$.

Let $d^n : C^n(L,M) \rightarrow C^{n+1}(L,M)$ be an $F$-homomorphism
defined by
 \begin{multline*}
(d^nf)(x_1, \dots , x_{n+1}): = [x_1,f(x_2,\dots,x_{n+1})]
+\sum\limits_{i=2}^{n+1}(-1)^{i}[f(x_1,
\dots, \widehat{x}_i, \dots , x_{n+1}),x_i]\\
+\sum\limits_{1\leq i<j\leq {n+1}}(-1)^{j+1}f(x_1, \dots,
x_{i-1},[x_i,x_j], x_{i+1}, \dots , \widehat{x}_j, \dots
,x_{n+1}),
\end{multline*}
 where $f\in C^n(L,M)$ and $x_i\in L$. Since the derivative
operator $d=\sum\limits_{i \geq 0}d^i$ satisfies the property
$d\circ d = 0$, the $n$-th cohomology group is well defined and
$$HL^n(L,M) = ZL^n(L,M)/ BL^n(L,M),$$
where the elements $ZL^n(L,M)$ and $BL^n(L,M)$) are called {\it
$n$-cocycles} and {\it $n$-coboundaries}, respectively.

The elements $f\in BL^2(L,L)$ and $\varphi \in ZL^2(L,L)$ are
defined as follows
\begin{equation}\label{E.B2} f(x,y) = [d(x),y] + [x,d(y)] -
d([x,y]) \ \mbox{for some linear map} \ d\end{equation} and
\begin{equation}\label{E.Z2}(d^2\varphi)(x,y,z)=[x,\varphi(y,z)]
 - [\varphi(x,y), z] + [\varphi(x,z), y] + \varphi(x, [y,z]) - \varphi([x,y],z) + \varphi([x,z],y)=0. \end{equation}

Usually a $2$-cocycle is called infinitesimal deformation.

{\it A deformation of a Leibniz algebra} $L$ is a one-parameter
family $L_t$ of Leibniz algebras with the bracket $$\mu_t = \mu_0
+ t\varphi_1 + t^2\varphi_2 + \cdots,$$ where $\varphi_i$ are
2-cochains, i.e., elements of $\Hom(L\otimes L, L)
= C^2(L, L)$.

Two deformations $L_t, \ L'_t$ with corresponding laws $\mu_t, \
\mu'_t$ are {\it equivalent} if there exists a linear automorphism
$f_t = id + f_1 t + f_2 t^2 + \cdots$ of $L$, where $f_i$ are
elements of $C^1(L, L)$ such that the following equation holds
$$\mu'_t(x, y) = f_t^{-1}(\mu_t(f_t(x), f_t(y))) \ \ \text{for} \  x, y \in L.$$

The Leibniz identity for the algebras $L_t$ implies that the
2-cochain $\varphi_1$ is an infinitesimal deformation, i.e. $d^2\varphi_1 =
0$. If $\varphi_1$ vanishes identically, then the first non vanishing
$\varphi_i$ is an infinitesimal deformation.

If $\mu'_t$ is an equivalent deformation with cochains
$\varphi_i'$, then $\varphi_1' -\varphi_1 = d^1f_1$, hence every
equivalence class of deformations defines uniquely an element of
$HL^2(L, L)$.

Note that the linear integrable deformation $\varphi$ satisfies the
condition
\begin{equation}\label{E:2.1}
\varphi(x,  \varphi(y, z)) -  \varphi(\varphi(x, y), z) +
\varphi(\varphi(x, z), y) = 0.
\end{equation}

The linear reductive group
$GL_n(F)$ acts on $Leib_n$ via change of basis, i.e.,
$$(g*\lambda)(x,y)=g \Big(\lambda \big(g^{-1}(x),g^{-1}(y) \big) \Big), \quad  g \in GL_n(F), \  \lambda \in Leib_n.$$

The orbits $\Orb(-)$ under this action are the  isomorphism
classes of algebras. Recall, Leibniz algebras with open orbits are called {\it rigid}.
Note that solvable (respectively, nilpotent)
Leibniz algebras of the same dimension also  form an invariant
subvariety of the variety of Leibniz algebras under the mentioned
action.
We give a definition of degeneration.

\begin{defn}  It is said that an algebra $\lambda$
degenerates to an algebra $\mu$, if $\Orb(\mu)$ lies in the
Zariski closure of $\Orb(\lambda)$, $\overline{\Orb(\lambda)}$. We
denote this by $\lambda \rightarrow \mu$.
\end{defn}

In the case of the field $\mathbb{F}$ be the complex numbers $\mathbb{C}$, we give an equivalent
definition of degeneration.

\begin{defn} Let $g : (0, 1] \rightarrow GL_n(V)$ be a continues mapping. We construct a
parameterized family of the Leibniz algebras $g_{t} = (V,
[-,-]_{t}), t\in (0, 1]$ isomorphic to $L.$ For each $t$ the new
Leibniz bracket $[-,-]_{t}$ on $V$ is defined via the old one as
follows: $[x,y]_{t} = g_{t}[g_{t}^{-1}(x),g_{t}^{-1}(y)],$ $\forall x,
y \in V.$ If for any $x,y \in V$ there exists the limit
$$\lim_{t \rightarrow +0}[x,y]_{t} =\lim_{t\rightarrow +0}g_{t}[g_{t}^{-1}(x),g_{t}^{-1}(y)] =: [x,y]_0,$$ then
$[-,-]_0$ is a well-defined Leibniz bracket. The Leibniz algebra
$L_0=(V, [-, -]_0)$ is called a degeneration of the algebra $L.$
\end{defn}

For a Leibniz algebra $L$ consider the following central lower
series:
$$
L^1=L,\quad L^{k+1}=[L^k,L^1], \quad k \geq 1.
$$


\begin{defn} An $n$-dimensional Leibniz algebra is said to be filiform if
$\dim L^i=n-i, \ 2\leq i \leq n$.
\end{defn}

Now let us define a natural graduation for a filiform Leibniz
algebra.

\begin{defn} Given a filiform Leibniz algebra $L$, put
$L_i=L^i/L^{i+1}, \ 1 \leq i\leq n-1$, and $gr(L) = L_1 \oplus
L_2\oplus\dots \oplus L_{n-1}$. Then $[L_i,L_j]\subseteq L_{i+j}$
and we obtain the graded algebra $gr(L)$. If $gr(L)$ and $L$ are
isomorphic, then we say that the algebra $L$ is naturally graded.
\end{defn}

In the following theorem we resume the classification of the
naturally graded filiform Leibniz algebras given in
\cite{Ayup} and \cite{Ver}.

\begin{thm} \label{thm2.4} Any complex $n$-dimensional naturally graded filiform
Leibniz algebra is isomorphic to one of the following pairwise non
isomorphic algebras:
$$\begin{array}{l}F_n^1: \ [x_i,x_1]=x_{i+1}, \  2\leq i \leq {n-1},\\
F_n^2: \ [x_i,x_1]=x_{i+1}, \  1\leq i \leq {n-2},\\
F_n^3(\alpha): \ \left\{\begin{array}{lll} [x_i,x_1]=-[x_1,x_i]=x_{i+1}, &
2\leq i \leq {n-1},\\[1mm]
[x_i,x_{n+1-i}]=-[x_{n+1-i},x_i]=\alpha (-1)^{i+1}x_n, & 2\leq
i\leq n-1.
\end{array} \right.\end{array}$$
where $\alpha\in\{0,1\}$
for even $n$ and $\alpha=0$ for odd $n.$
\end{thm}

The following theorem decomposes  all $n$-dimensional
filiform Leibniz algebras into  three families of algebras.

\begin{thm} \label{th2.5} \cite{OmRa} Any complex $n-$dimensional filiform
Leibniz algebra admits a basis $\{x_1, x_2, \dots, x_n\}$ such
that the table of multiplication of the algebra have one of the
following forms:

$F_1=\left\{\begin{array}{ll}
[x_i,x_1]=x_{i+1}, & \  2\leq i \leq {n-1},\\[1mm]
[x_1,x_2]=\theta x_n, & \\[1mm]
[x_j,x_2]=\alpha_4x_{j+2} + \alpha_5x_{j+3}+ \dots +
\alpha_{n+2-j}x_n, & \ 2\leq j \leq {n-2},
\end{array} \right.$ \\[1mm]

$F_2=\left\{\begin{array}{ll}
[x_i,x_1]=x_{i+1}, & \  1\leq i \leq {n-2},\\[1mm]
[x_j,x_n]=\beta_3x_{j+2} + \beta_4x_{j+3}+\dots+
\beta_{n-j}x_{n-1}, & \
1\leq j \leq {n-3},\\[1mm]
[x_n,x_n]=\gamma x_{n-1}, &
\end{array} \right.$ \\

$F_3=\left\{\begin{array}{lll} [x_i,x_1]=x_{i+1}, &
2\leq i \leq {n-1},\\[1mm]
[x_1,x_i]=-x_{i+1}, & 3\leq i \leq {n-1}, \\[1mm]
[x_1,x_1]=\theta_1x_n, &   \\[1mm]
[x_1,x_2]=-x_3+\theta_2x_n, & \\[1mm]
[x_2,x_2]=\theta_3x_n, &  \\[1mm]
[x_i,x_j]=-[x_j,x_i] \in lin<x_{i+j+1}, x_{i+j+2}, \dots , x_n>, &
2\leq i < j \leq {n-1},\\[1mm]
[x_i,x_{n+1-i}]=-[x_{n+1-i},x_i]=\alpha (-1)^{i+1}x_n, & 2\leq
i\leq n-1.
\end{array} \right.$ \\
where  $\alpha\in\{0,1\}$ for
even $n$ and $\alpha=0$ for odd $n.$
\end{thm}

In \cite{KhudOmir} we obtain that any single-generated Leibniz algebra has the following
multiplication:
$$\widetilde{\mu}(\alpha_2, \alpha_3,
\dots, \alpha_n)= \begin{cases}[x_i, x_1] = x_{i+1}, & 1\leq i
\leq n-1,\\
[x_n, x_1] = \sum\limits_{k=2}^n\alpha_{k}x_k.\end{cases}$$

Note that any algebra of the family $\widetilde{\mu}(\alpha_2, \alpha_3,
\dots, \alpha_n)$ is linear integrable deformation of the algebra $NF_n.$

Let us introduce denotation
$$X=\overline{\bigcup\limits_{\alpha_2, \dots, \alpha_n}Orb(\widetilde{\mu}(\alpha_2, \alpha_3, \dots,
\alpha_n))}.$$

\begin{thm} \cite{KhudOmir} \label{th2.10}
$X$ is an irreducible component of the variety $Leib_n$.
\end{thm}

\section{Deformations of the naturally graded filiform Leibniz algebras}

In this section we calculate infinitesimal deformations of the
naturally graded filiform Leibniz algebras.

\subsection{Infinitesimal deformations of the algebra $F_n^{1}$}

\

In order to achieve the purpose of the subsection we need the matrix form of a derivation of the filiform Leibniz algebra $F_n^{1}$ \cite{Cas2}:
\begin{equation} \label{E:Der1}\begin{pmatrix}
\alpha_1& 0&0&0&\dots&0&\alpha_n\\
0& \beta_2&\beta_3&\beta_4&\dots&\beta_{n-1}& \beta_n\\
0& 0&\alpha_1+\beta_2&\beta_3&\dots&\beta_{n-2}& \beta_{n-1}\\
0& 0&0&2\alpha_1+\beta_2&\dots&\beta_{n-3}&\beta_{n-2}\\
\vdots&\vdots&\vdots&\vdots&\dots&\vdots&\vdots\\
0&0&0&0&\dots&0&(n-2)\alpha_1+\beta_2
\end{pmatrix}
\end{equation}

Due to \eqref{E.B2} it is easy to see that $\dim BL^2(F_n^1,
F_n^1) = n^2 - n -1$.

The following proposition presents the general form of the Leibniz infinitesimal deformation of the algebra $F_n^1.$

\begin{prop} \label{4.4} An arbitrary infinitesimal deformation $\varphi$ of $F_n^1$
has the following form:
\[\begin{cases}\varphi(x_1, x_1) = \sum\limits_{k=2}^{n} \alpha_{1,k}x_k, \quad
\varphi(x_j, x_1) = \sum\limits_{k=1}^{n} \alpha_{j,k}x_k, &   2 \leq j \leq n-1, \\
\varphi(x_n, x_1) = \sum\limits_{k=2}^{n} \alpha_{n,k}x_k, \quad \varphi(x_1, x_2) = \gamma_1x_1 + \gamma_n x_n, & \\
\varphi(x_i, x_2) = ((i-2)\gamma_1+\beta_2)x_i+\sum\limits_{k=3}^{n+2-i} \beta_kx_{k+i-2}, &   2 \leq i \leq n,\\
\varphi(x_i, x_3) = -(\alpha_{2,1}+\gamma_1)x_{i+1}, & 2 \leq i \leq n-1, \\
\varphi(x_i, x_{j+1}) = -\alpha_{j,1}x_{i+1}, & 2 \leq i \leq n-1, \  3 \leq j \leq n-1.\end{cases}\]
\end{prop}

\begin{proof} Using the property of infinitesimal deformations for $(d^2\varphi )(x_i, x_j, x_k) = 0$ with $2 \leq j,k \leq n$, we obtain $[x_i, \varphi(x_j, x_k)] =0,$ which implies $\varphi(x_j, x_k) \in<x_2, x_3, \dots, x_n>.$

Similarly, the equation $(d^2\varphi )(x_i, x_1, x_1) = 0$ leads to $[x_i, \varphi(x_1, x_1)] =0,$ consequently we have $\varphi(x_1, x_1)\in<x_2, x_3, \dots, x_n>.$

From the condition  $(d^2\varphi )(x_i, x_j, x_1) = 0$ with $2 \leq j \leq n$, we derive
\begin{equation}\label{E:3.1}
[x_i,  \varphi(x_j, x_1)] - [\varphi(x_i, x_j), x_1] + \varphi(x_i, [x_j,x_1])+ \varphi([x_i,x_1], x_j)=0.
\end{equation}

Similarly, from the condition  $(d^2\varphi )(x_i, x_1, x_k) = 0$ with $2 \leq k \leq n$, we have
\begin{equation}\label{E:3.2}
[x_i,  \varphi(x_1, x_k)] + [\varphi(x_i, x_k), x_1] -
\varphi([x_i,x_1], x_k)=0.
\end{equation}

The equality \eqref{E:3.2} with $i=1, k=2$ deduce $[\varphi(x_1, x_2),
x_1] = \varphi([x_1,x_1], x_2) - [x_1,  \varphi(x_1, x_2)]=0,$
hence we can assume $\varphi(x_1, x_2) = \gamma_1x_1+\gamma_nx_n$ for some parameters $\gamma_1, \gamma_2$.

From equality \eqref{E:3.1} with $i=1$, $2 \leq j \leq n-1,$ we have
$\varphi(x_1, x_{j+1}) = [\varphi(x_1, x_j), x_1]=0.$

Summarizing equalities \eqref{E:3.1} and \eqref{E:3.2} we obtain
\begin{equation}\label{E:3.3}\begin{cases}\varphi(x_i, x_3) = - [x_i,\varphi(x_1, x_{2})+\varphi(x_2, x_{1})],&\\
\varphi(x_i, x_{j+1}) = - [x_i,\varphi(x_j, x_{1})], & 3 \leq j \leq n-1,\\
[x_i,\varphi(x_n, x_{1})] =0.& \end{cases}\end{equation}

We set
$$\varphi(x_j, x_1) = \sum\limits_{k=1}^{n} \alpha_{j,k}x_k, \   1 \leq j \leq n, \quad
\varphi(x_2, x_2) = \sum\limits_{k=2}^{n} \beta_{k}x_k. $$

Applying the equations \eqref{E:3.1}, \eqref{E:3.2}  and
\eqref{E:3.3} we derive $a_{1,1} = a_{n,1} =0$ and
$$\varphi(x_i, x_2) = ((i-2)\gamma_1+\beta_2)x_i+\sum\limits_{k=3}^{n+2-i} \beta_kx_{k+i-2}, \   2 \leq i \leq n,$$
$$ \varphi(x_i, x_3) = -(\alpha_{2,1}+\gamma_1)x_{i+1}, \ 2 \leq i \leq n-1, \quad
\varphi(x_i, x_{j+1}) = -\alpha_{j,1}x_{i+1}, \ 2 \leq i \leq n-1,
\  3 \leq j \leq n-1.$$
\end{proof}

Using Proposition \ref{4.4} we indicate a basis
of the space $ZL^2(F_n^1, F_n^1)$.

\begin{thm}\label{thm3.3} The following cochains:
\begin{align*} \varphi_{j,1}(2 \leq j \leq n -1) &: \left\{\begin{array}{ll}\varphi_{j,1}(x_j, x_1) =
x_1,& \\ \varphi_{j,1}(x_i, x_{j+1}) =-x_{i+1},& 2 \leq i \leq
n-1,\end{array}\right.\\
 \varphi_{j,k}(1 \leq j \leq n, \ 2 \leq k \leq n) &: \left\{\varphi_{j,k}(x_j, x_1) = x_k,\right. \\
\psi_j(2 \leq j \leq n) &: \left\{\psi_j(x_i, x_2) = x_{j+i-2},\right. \  2 \leq i \leq n-j+2,\\
 \xi_{1}&:\left\{\begin{array}{ll}\xi_{1}(x_1, x_2) =
x_1, \\ \xi_{1}(x_i, x_2) =(i-2)x_{i},& 3 \leq i \leq n,
\\ \xi_{1}(x_i, x_{3}) =-x_{i+1},& 2 \leq i \leq
n-1,\end{array}\right.\\ \xi_{2} &: \left\{\xi_{2}(x_1, x_2) =
x_n\right.\end{align*} form a basis of the space $ZL^2(F_n^1,
F_n^1)$.
\end{thm}

\begin{cor}
$\dim(ZL^2(F_n^1,F_n^1)) = n^2+n-1$.
\end{cor}

Below, we describe a basis of the subspace $BL^2(F_n^1,F_n^1)$ in
terms of $\varphi_{j,k},$ $\psi_j,$ $\xi_{1}$ and $\xi_{2}.$

\begin{prop}\label{pr4} The cocycles
$$\eta_{j,k}: \begin{cases}\eta_{1,k-1} = \varphi_{1,k} , & 3 \leq k \leq n, \\
\eta_{2,1} = \psi_{3} , & \\
\eta_{j,1} = \varphi_{j-1,1} , & 3 \leq j \leq n, \\
\eta_{j,k} = \varphi_{j-1,k} , & 3 \leq j \leq k \leq n, \\
\eta_{j,k} = \varphi_{j-1,k} - \varphi_{j,k+1} , & 3 \leq k < j \leq n, \\
\end{cases}$$ form a basis of $BL^2(F_n^1,F_n^1)$.
\end{prop}
\begin{proof} Consider the endomorphisms $f_{j,k}$ defined as follows:
$$\begin{cases} f_{2,1}(x_2) = x_1, & \\ f_{1,k}(x_1) = x_k, &  2 \leq k \leq n-1,\\
f_{j,k}(x_j) = x_k, & 3 \leq j \leq n, \ 1 \leq k \leq n.\end{cases}$$

According to \eqref{E:Der1} it implies that $f_{j,k}$ are complemented linear maps to derivations in
$C^1(F_n^1,F_n^1)$. Therefore, $d^1f_{j,k}$ form a basis of the space
$BL^2(F_n^1,F_n^1),$ where $d^1f_{j,k} =
f_{j,k}([x,y]) - [f_{j,k}(x),y] - [x, f_{j,k}(y)]$.

It should be noted that
$$\begin{cases}d^1f_{1,k} =  -\varphi_{1,k+1}, & 2 \leq k \leq n-1,\\
d^1f_{2,1} =  -\psi_3,\\
d^1f_{j,1} =  \varphi_{j-1,1}, & 3 \leq j \leq n,\\
d^1f_{j,k} =  \varphi_{j-1,k} - \varphi_{j,k+1}, & 3 \leq j \leq n, \  2 \leq k \leq n-1,\\
d^1f_{j,n} =  \varphi_{j-1,n}, & 3 \leq j \leq n.\\
\end{cases}$$

From the condition $d^1f_{j,k} + d^1f_{j+1,k+1} + \dots +
d^1f_{n+j-k,n} = \varphi_{j-1,k}$ for $3 \leq j \leq k \leq n$, we
conclude that the maps $\eta_{j,k}$ form a basis of
$BL^2(F_n^1,F_n^1)$.
\end{proof}

\begin{cor} \label{cor34} The adjoint classes $\overline{\psi_2}, \overline{\xi_1}, \overline{\xi_2}$,
$\overline{\varphi_{1,2}}$, $\overline{\varphi_{n,k}}$ ($2 \leq k \leq n$) and $\overline{\psi_{j}}$ ($4 \leq j \leq n$) form a basis
of $HL^2(F_n^1,F_n^1)$. Consequently, $\dim HL^2(F_n^1,F_n^1) = 2n$.
\end{cor}

Since every non-trivial equivalence class of deformations defines
uniquely an element of $HL^2(L, L)$, due to Corollary \ref{cor34}
it is sufficient to consider a linear deformation $$\mu_t =
F_{n}^1 + t\varphi,$$ where $\varphi=
c_1\xi_1+c_2\xi_2+a_{1}\varphi_{1,2}+\sum\limits_{k=2}^na_{k}\varphi_{n,k}
+b_{2}\psi_{2}+\sum\limits_{k=4}^nb_{k}\psi_{k}$.

If $t\neq 0,$ then we can assume  $t=1$ and the linear deformation $\mu_1$ we shall denote by $\mu$:
\begin{equation}\label{E:3.4} \mu: \ \left\{\begin{array}{ll}
[x_1,x_1]=a_1x_2, & \\[1mm]
[x_i,x_1]=x_{i+1}, & \  2\leq i \leq {n-1},\\[1mm]
[x_n,x_1]=\sum\limits_{k=2}^na_{k}x_k,& \\[1mm]
[x_1,x_2]=c_1x_1 +c_nx_n, & \\[1mm]
[x_i,x_2]=((i-2)c_1+b_2)x_{i} + \sum\limits_{k=4}^{n+2-i} b_kx_{k+i-2}, & \  2\leq i \leq {n},\\[1mm]
[x_i,x_3]=-c_1x_{i+1}, & \  2\leq i \leq {n-1}.\\[1mm]
\end{array} \right.\end{equation}

In the next proposition we clarify under which conditions on
parameters $a_i, b_i, c_1$ and $c_n$ the algebra of the family
$\mu$ is a Leibniz algebra.

\begin{prop} Linear integrable deformation of the algebra
$F_n^1$ consist of the first class of filiform Leibniz algebra $F_1$ and following Leibniz algebras
$$\lambda(a_1, \dots, a_n):  \left\{\begin{aligned}
{}[x_1,x_1]&=a_1x_2, \\
[x_i,x_1]&=x_{i+1}, && 2\leq i \leq {n-1},\\
[x_n,x_1]&=\sum\limits_{k=2}^na_{k}x_k,
\end{aligned}\right.\ R: \left\{\begin{aligned}
{}[x_i,x_1]& = x_{i+1}, && 2\leq i\leq n-1,\\
[x_1,x_2]& = x_1, \\
[x_i,x_2]& =(i-2)x_i, && 3\leq i\leq n,\\
[x_i,x_3]& = -x_{i+1}, && 2\leq i\leq n-1.\\
\end{aligned}\right. $$
\end{prop}

\begin{proof} Verifying Leibniz identity for the algebra $\mu$ we obtain the following restrictions:
$$b_2=0, \quad c_1a_{k} =0, \quad c_na_{k}=0, \quad
b_ia_{k} =0, \quad  4 \leq i \leq n, \ 1 \leq k \leq n.$$

If $a_i \neq 0$ for some $i,$ then we deduce $c_1 = c_n =0$,
$b_i=0$, $4 \leq i \leq n$. So, the family of algebras $\lambda(a_1, a_2, \dots, a_n)$ is obtained.

If $a_i  = 0$ for all $i,$ then table of multiplication of the family $\mu$ have the form:
$$\left\{\begin{array}{ll}
[x_i,x_1]=x_{i+1}, & \  2\leq i \leq {n-1},\\[1mm]
[x_1,x_2]=c_1x_1 +c_nx_n, & \\[1mm]
[x_i,x_2]=(i-2)c_1x_{i} + \sum\limits_{k=4}^{n+2-i} b_kx_{k+i-2}, & \  2\leq i \leq {n},\\[1mm]
[x_i,x_3]=-c_1x_{i+1}, & \  2\leq i \leq {n-1}.
\end{array} \right.$$

In the case of $c_1 = 0$ we get the family of filiform Leibniz algebras $F_1.$

If $c_1 \neq 0,$ then taking the basis transformation in the
following form:
$$
x_1' = x_1 +\frac {c_n}{c_1(3-n)} x_{n}, \quad x_i'=\frac 1 {c_1}
x_i+\sum\limits_{j=i+2}^{n}A_{j-i+2}x_j,\
 2\leq i\leq n-2, \quad x_{n-1}'=\frac 1 {c_1}
x_{n-1},\quad x_n'=\frac 1 {c_1} x_n,
$$
with
$$ A_4=-\frac {b_4}{2c_1^2}, \quad A_5=-\frac {b_5}{3c_1^2}, \quad
A_i=- \frac{1}{(i-2)c_1}\Big(\frac
{b_i}{c_1}+\sum\limits_{j=4}^{i-2}A_jb_{i+2-j}\Big),\qquad 6\leq
i\leq n,$$ we obtain the algebra $R.$ \end{proof}

Below we establish in which irreducible component belongs the family of algebras $\lambda$.

\begin{prop}
$ \lambda(a_1, a_2, \dots, a_n) \in X$ for any values of parameters
$a_i$.
\end{prop}
\begin{proof} If in the family $\lambda$ the parameter $a_1\neq 0$, then by taking the change of basis elements as follows: $x_1'=x_1, \ x_i'=a_1x_i,$ $2 \leq i \leq n$, we can assume $a_1=1$ and the family $\lambda(1, a_2, \dots, a_n)$ is nothing else but $\widetilde{\mu}(\alpha_2, \alpha_3, \dots, \alpha_n).$

If $a_1 =0,$ then in the case of $a_2 \neq 0 $ by changing of basis in the following way:
$$x_1' = x_1 +\frac 1 {a_2}(x_n -\sum\limits_{k=3}^na_kx_{k-1}), \quad x_i'=x_i, \ 2 \leq i \leq n,$$
we have $\lambda(0, a_2, \dots, a_n) \simeq
\widetilde{\mu}(\alpha_2, \alpha_3, \dots, \alpha_n).$

Let suppose $a_1=a_2=0,$ then by choosing the transformations $g_t$ as follows $g_t(x_1) =x_1, \ g_t(x_i) = tx_i,\ 2 \leq i
\leq n,$ we derive
$$\lim_{t\rightarrow 0} g_t\ast \widetilde{\mu}(0, \alpha_3,
\dots, \alpha_n)  = \lambda(0, 0, a_3, \dots, a_n),$$ which
implies $\lambda(0,0,a_3, \dots, a_n) \in X.$  \end{proof}

The interesting properties of the algebra $R$ are given in the following assertions.

\begin{prop} \label{prop38} Any derivation of the algebra $R$ has the matrix form:
 $$
\begin{pmatrix}
\alpha& 0&0&0&\dots&0&0\\
0& 0&\beta&0&\dots&0& 0\\
0& 0&\alpha&\beta&\dots&0& 0\\
0& 0&0&2\alpha&\dots&0&0\\
\vdots&\vdots&\vdots&\vdots&\dots&\vdots&\vdots
\\
0&0&0&0&\dots&(n-3)\alpha&\beta\\
0&0&0&0&\dots&0&(n-2)\alpha
\end{pmatrix}.$$
\end{prop}

Taking into account that operators of right multiplications $R_{x_1}$ and $R_{x_2}$ of the algebra $R$ are
linear independent inner derivations, we conclude that any derivation of the algebra $R$ is inner.

\begin{prop} \label{prop39} Any infinitesimal deformation of the algebra $R$ has the following form:
$$\begin{cases}\varphi(x_1, x_1) = a_{1,1}x_1 + a_{1,1}x_3+\sum\limits_{k=4}^{n} a_{1,k}x_k, \\
\varphi(x_i, x_1) = \sum\limits_{k=1}^{n} a_{i,k}x_k, & 2 \leq i \leq n-1,\\
\varphi(x_n, x_1) = -\sum\limits_{k=3}^{n-1} \sum\limits_{j=1}^{k-2}a_{n-j,k-j}x_k +
\Big(\frac {(n-1)(n-2)} 2a_{1,1} -\sum\limits_{k=2}^{n-1} a_{k,k}\Big)x_n,\\
\varphi(x_1, x_2) = b_{1,1}x_1 - a_{1,1}x_2+\sum\limits_{k=4}^{n-1} (n-3)a_{1,k+1}x_k, \\
\varphi(x_2, x_2) = b_{2,1}x_1 + b_{2,1}x_3+\sum\limits_{k=4}^{n} b_{2,k}x_k,\\
\varphi(x_3, x_2) = a_{2,2}x_2 + b_{1,1}x_3+\sum\limits_{k=4}^{n} (b_{2,k-1}-(k-3)a_{2,k})x_k,\\
\varphi(x_i, x_2) = (i-3)a_{i-1,1}x_1 + (i-2)b_{1,1}x_i-
\Big(\frac {(i-2)(i-3)} 2a_{1,1} -\sum\limits_{k=2}^{i-1} a_{k,k}\Big)x_{i-1}+\\
+\sum\limits_{k=2}^{i-2} (i-k)\sum\limits_{j=1}^{k-1}a_{i-j,k+1-j}x_k+
\sum\limits_{k=i+1}^{n} \Big(b_{2,k+2-i}+(i-k)\sum\limits_{j=2}^{i-1}a_{j,k+1+j-i}\Big)x_k, & 4 \leq i \leq n,
\\ \varphi(x_1, x_3) = -a_{2,2}x_1 - a_{1,1}x_3-\sum\limits_{k=4}^{n} a_{1,k}x_k, \\
\varphi(x_2, x_3) = -a_{2,1}x_1 - a_{2,2}x_2-(b_{1,1}+a_{2,1})x_3 -
\sum\limits_{k=4}^{n} a_{2,k}x_k, \\
\varphi(x_i, x_3) = (i-2)(a_{1,1}-a_{2,2})x_i + (a_{2,3}-a_{2,1}-b_{1,1})x_{i+1}-
\sum\limits_{k=1}^{n} a_{i,k}x_k, & 3 \leq i \leq n-1,\\
\varphi(x_n, x_3) = \sum\limits_{k=3}^{n-1} \sum\limits_{j=1}^{k-2}a_{n-j,k-j}x_k -
\Big(\frac {(n-2)(n-3)} 2a_{1,1} +(n-3)a_{2,2}-\sum\limits_{j=4}^{n-1} a_{k,k}\Big)x_n, \\
\varphi(x_1, x_j) = -a_{j-1,2}x_1, & 4 \leq j \leq n,\\
\varphi(x_2, x_j) = (a_{j-2,2}-a_{j-1,1}+a_{j-1,3})x_3, & 4 \leq j \leq n,\\
\varphi(x_i, x_j) = -(i-2)a_{j-1,2}x_i+ (a_{j-2,2}-a_{j-1,1}+a_{j-1,3})x_{i+1}, & 3 \leq i \leq n-1,  \ 4 \leq j \leq n,\\
\varphi(x_n, x_j) = -(n-2)a_{j-1,2}x_n, &   \ 4 \leq j \leq n.\\
\end{cases}$$
\end{prop}

\begin{cor} The algebra $R$ is rigid.
\end{cor}
\begin{proof} Due to Proposition \ref{prop38} we have $\dim Der R =2$. Therefore, $\dim BL^2(R, R) = n^2
-2.$ From Proposition \ref{prop39} we conclude $\dim ZL^2(R, R) = n^2 -2,$ hence $HL^2(R, R) = 0.$ Applying the result of the paper \cite{Balavoine} on rigidity of Leibniz algebras which satisfy the condition $HL^2(R, R) = 0$ we complete the proof.
\end{proof}

\subsection{Infinitesimal deformations of the algebra $F_n^{2}$}

\

Further we shall use the result of \cite{Cas2} on description of derivations of the filiform Leibniz algebra $F_n^{2}.$ Namely, any derivation of $F_n^{2}$ has the following matrix form:
\begin{equation} \label{E.Der2}
\begin{pmatrix}
\alpha_1& \alpha_2&\alpha_3&\dots&\alpha_{n-1}&\alpha_n\\
0& 2\alpha_1&\alpha_2&\dots&\alpha_{n-2}& 0\\
0& 0&3\alpha_1&\dots&\alpha_{n-2}& 0\\
\vdots&\vdots&\vdots&\dots&\vdots&\vdots\\
0& 0&0&\dots&(n-1)\alpha_1&0\\
0&0&0&\dots&\beta_1&\beta_2
\end{pmatrix}.
\end{equation}

This matrix form of derivations implies $\dim Der(F_n^2) = n +2$ and $\dim BL^2(F_n^2, F_n^2) = n^2 - n -2$.

\begin{prop} \label{4.5} An arbitrary infinitesimal deformation $\varphi$ of $F_n^2$
has the following form:
\[\begin{cases}\varphi(x_j, x_1) = \sum\limits_{k=1}^{n} \alpha_{j,k}x_k, &   1 \leq j \leq n-2, \\
\varphi(x_{n-1}, x_1) = \sum\limits_{k=2}^{n} \alpha_{n-1,k}x_k, \quad
\varphi(x_n, x_1) = \sum\limits_{k=1}^{n} \alpha_{n,k}x_k, \\
\varphi(x_i, x_{j+1}) = -\alpha_{j,1}x_{i+1}, & 1 \leq i \leq n-2, \  1 \leq j \leq n-2, \\
\varphi(x_1, x_n) = -\alpha_{n,1}x_1+\sum\limits_{k=2}^{n} \beta_kx_{k}, &\\
\varphi(x_i, x_n) = -i\alpha_{n,1}x_i+\sum\limits_{k=2}^{n-i} \beta_kx_{k+i-1}, &   2 \leq i \leq n-1,\\
\varphi(x_n, x_n) = \gamma_1x_{n-1} + \gamma_n x_n. &
\end{cases}\]
\end{prop}

\begin{proof} The proof of this proposition is carrying out by applying similar arguments as in the
proof of Proposition \ref{4.4}.
\end{proof}

Using the assertion of Proposition \ref{4.5} we indicate a
basis of the space $ZL^2(F_n^2, F_n^2)$.

\begin{thm}\label{thm2} The following cochains:
\begin{align*} \varphi_{j,1}(1 \leq j \leq n-2) &: \left\{\begin{array}{ll}\varphi_{j,1}(x_j, x_1) =
x_1,& \\ \varphi_{j,1}(x_i, x_{j+1}) =-x_{i+1},& 1 \leq i \leq
n-2,\end{array}\right.\\ \varphi_{j,k}(1 \leq j \leq n, \ 2 \leq k
\leq n) &: \left\{\varphi_{j,k}(x_j, x_1) = x_k,\right. \\
\psi_1 &: \left\{\begin{array}{ll}\psi_1(x_n, x_1) = x_1,& \\
\psi_1(x_i, x_n) =-ix_{i},& 1 \leq i \leq
n-1,\end{array}\right.\\
\psi_j(2 \leq j \leq n-1) &: \left\{\psi_j(x_i, x_n) = x_{j+i-1},  \ 1 \leq i \leq n-j,\right.\\
\psi_n &: \left\{\psi_n(x_1, x_n)=x_{n},\right. \\
\psi_{n+1} &: \left\{\psi_{n+1}(x_n, x_n)=x_{n-1},\right. \\
\psi_{n+2} &: \left\{\psi_{n+2}(x_n,
x_n)=x_{n}.\right.\end{align*} form a basis of
$ZL^2(F_n^2,F_n^2)$.
\end{thm}

\begin{cor}
$\dim(ZL^2(F_n^2,F_n^2)) = n^2+n$.
\end{cor}

Below we describe a basis of the subspace $BL^2(F_n^2,F_n^2)$ by means of $\varphi_{j,k},$ $\psi_j.$

\begin{prop}\label{pr4} The cocycles
$$\eta_{j,k}: \begin{cases}
\eta_{j,1} = \varphi_{j-1,1} - \varphi_{j,2} , & 2 \leq j \leq n-1, \\
\eta_{j,k} = \varphi_{j-1,k} , & 2 \leq j \leq k \leq n-1, \\
\eta_{j,k} = \varphi_{j-1,k} - \varphi_{j,k+1} , & 2 \leq k < j \leq n-1, \\
\eta_{j,n} = \varphi_{j-1,n} , & 2 \leq j \leq n-1, \\
\eta_{n,1} = \varphi_{n,2}+\psi_2,\\
\eta_{n,k} = \varphi_{n,k+1} , & 2 \leq k \leq n-2,
\end{cases}$$ form a basis of $BL^2(F_n^2,F_n^2)$.
\end{prop}

\begin{cor} \label{cor3.15} The adjoint classes $\overline{\varphi_{n,n}}$,
$\overline{\varphi_{n-1,k}}$ $(2 \leq k \leq n)$ and
$\overline{\psi_{j}}$ $(1 \leq j \leq n+2)$ form a basis of
$HL^2(F_n^2,F_n^2)$. Consequently, $\dim HL^2(F_n^2,F_n^2) =
2n+2$.
\end{cor}

In the next proposition we clarify that basis elements of
$ZL^2(F_n^2,F_n^2)$ satisfies the condition \eqref{E:2.1}.

\begin{prop} The infinitesimal deformations $\varphi_{j,k} \ (1 \leq j \leq n, 2 \leq k \leq n),$
$\psi_{j} \ (1 \leq j \leq n-1)$  and $\psi_{n+1}$ satisfy the
condition \eqref{E:2.1}, however $\psi_{n},$ $\psi_{n+2},$ and
 $\varphi_{j,1} \ (1 \leq j \leq n-2)$ do not satisfy the
condition \eqref{E:2.1}.
\end{prop}

\begin{proof} The proof of this proposition is straightforward.
\end{proof}

Since every non-trivial equivalence class of deformations defines
uniquely an element of $HL^2(L, L)$, due to Corollary \ref{cor3.15}
it is sufficient to consider the linear deformation $$\nu_t = F_{n}^2 + t\varphi,$$ where
$\varphi= c_{1}\varphi_{n,n}+\sum\limits_{k=2}^na_{k}\varphi_{n-1,k}
+\sum\limits_{k=1}^{n+2}b_{k}\psi_{k}$.

Without loss of generality, for non-trivial linear deformation we can assume  $t=1.$

Then we have the table of multiplications
$$\nu: \ \left\{\begin{array}{ll}
[x_i,x_1]=x_{i+1}, & \  1\leq i \leq {n-2},\\[1mm]
[x_{n-1},x_1]=\sum\limits_{k=2}^na_kx_{k},\\[1mm]
[x_{n},x_1]=b_1x_1 + c_1x_n,\\[1mm]
[x_1,x_n]=-b_1x_{1} + \sum\limits_{k=2}^{n}b_kx_{k}, &\\[1mm]
[x_i,x_n]=-ib_1x_{i} + \sum\limits_{k=2}^{n-i}b_kx_{k+i-1}, & 2\leq i \leq {n-1},\\[1mm]
[x_n,x_n]=b_{n+1}x_{n-1} + b_{n+2}x_{n}. & \\[1mm]
\end{array} \right.$$

From the equalities $$0=[x_n,[x_n,x_n]] = [x_n, b_{n+1}x_{n-1} + b_{n+2}x_{n}] =
b_{n+2}(b_{n+1}x_{n-1} + b_{n+2}x_{n})$$ we get $b_{n+2}=0.$

\begin{prop} Any linear integrable deformation of the algebra
$F_n^2$ admits a basis $\{x_1, x_2, \dots, x_n\}$ such that its table of multiplication has the form of the families $F_1,$ $F_2,$ $\widetilde{\mu}(a_2, \dots, a_n),$
$\widetilde{\mu}(a_2, \dots, a_{n-1}) \oplus \mathbb{C}$ and
$$\nu_1(a_2, a_3, \dots a_{n-1}): \ \left\{\begin{array}{ll}
[x_i,x_1]=x_{i+1}, & 1\leq i \leq {n-2},\\[1mm]
[x_{n-1},x_1]=\sum\limits_{k=2}^{n-1}a_kx_{k}, & \sum\limits_{k=2}^{n-1}a_k=1,\\[1mm]
[x_{n},x_1]=x_n,
\end{array} \right.$$
$$\nu_2: \ \left\{\begin{array}{ll}
[x_i,x_1]=x_{i+1}, & \  1\leq i \leq {n-2},\\[1mm]
[x_{n},x_1]=x_1,\\[1mm]
[x_i,x_n]=-ix_{i}, & 1\leq i \leq {n-1},
\end{array} \right.$$
$$\nu_3(b_2, b_3, \dots, b_{n-1}): \ \left\{\begin{array}{ll}
[x_i,x_1]=x_{i+1}, & \  1\leq i \leq {n-2},\\[1mm]
[x_{n-1},x_1]=-x_{n-1},\\[1mm]
[x_{n},x_1]= -x_n,\\[1mm]
[x_1,x_n]=x_n + \sum\limits_{k=2}^{n-1}b_kx_{k}, &\\[1mm]
[x_i,x_n]=\sum\limits_{k=2}^{n-i}b_kx_{k+i-1}, & 2\leq i \leq {n-2},\\[1mm]
\end{array} \right.$$
where the first non-zero element of the vector
$(b_2, b_3, \dots , b_{n-1})$ can be assumed to be equal to 1,
$$\nu_4: \ \left\{\begin{array}{ll}
[x_i,x_1]=x_{i+1}, & \  1\leq i \leq {n-2},\\[1mm]
[x_{n-1},x_1]=-2x_{n-1},\\[1mm]
[x_{n},x_1]= -x_n,\\[1mm]
[x_1,x_n]=x_{n}, & \\[1mm]
[x_n,x_n]=x_{n-1}, & \\[1mm]
\end{array} \right.$$
$$\nu_5(a_2, a_3, \dots, a_{n-1}): \ \left\{\begin{array}{ll}
[x_i,x_1]=x_{i+1}, & \  1\leq i \leq {n-2},\\[1mm]
[x_{n-1},x_1]=\sum\limits_{k=2}^{n-1}a_kx_{k},\\[1mm]
[x_{n},x_1]= -x_n,\\[1mm]
[x_1,x_n]=x_{n}. &\\[1mm]
\end{array} \right.$$
\end{prop}

\begin{proof}
Note that $\{x_2, x_3, \dots, x_{n-1}\} \in Ann_r(\nu).$
If $x_n \in Ann_r(\nu),$ then we have  $b_k=0$, $1 \leq k \leq n+1.$

If $a_n\neq 0$ then we have the class of single generated algebras $\widetilde{\mu}(a_2, \dots, a_n)$.

If $a_n=0$ and $c_1 =0,$ then we have the split algebra $\widetilde{\mu}(a_2, \dots, a_{n-1})\oplus \mathbb{C}$.

If $a_n=0$ and $c_1 \neq 0,$ then by scaling the element $x_n$ we
can suppose $c_1=1.$ In fact, in the case of
$\sum\limits_{k=2}^{n-1}a_k\neq 1$ this algebra is also
single-generated (by the generator we can choose the element
$x_1+x_n$). It is easy to see that in the case of
$\sum\limits_{k=2}^{n-1}a_k = 1$ we obtain two-generated family of
algebras $\nu_1(a_2, a_3, \dots a_{n-1})$.

Now we consider the case of $x_n \notin Ann_r(\nu)$. It implies $a_n=
0$, $c_1 = -  b_n$ and the table of multiplication of $\nu$ has
the form:
$$\nu: \ \left\{\begin{array}{ll}
[x_i,x_1]=x_{i+1}, & \  1\leq i \leq {n-2},\\[1mm]
[x_{n-1},x_1]=\sum\limits_{k=2}^{n-1}a_kx_{k},\\[1mm]
[x_{n},x_1]=b_1x_1 -b_n x_n,\\[1mm]
[x_1,x_n]=-b_1x_{1} + \sum\limits_{k=2}^{n}b_kx_{k}, &\\[1mm]
[x_i,x_n]=-ib_1x_{i} + \sum\limits_{k=2}^{n-i}b_kx_{k+i-1}, & 2\leq i \leq {n-1},\\[1mm]
[x_n,x_n]=b_{n+1}x_{n-1}. & \\[1mm]
\end{array} \right.$$

Consider the Leibniz identity
$$[x_1,[x_n,x_1]] = [[x_1, x_n],x_1] - [[x_1,
x_1],x_n]= [-b_1x_{1} + \sum\limits_{k=2}^{n}b_kx_{k},x_1]  - [x_2, x_n]=$$
$$-b_1x_{2} + \sum\limits_{k=2}^{n-2}b_kx_{k+1} +b_{n-1}\sum\limits_{k=2}^{n-1}a_kx_{k} + b_n(b_1x_1 -b_n x_n) -
(-2b_1x_{2} + \sum\limits_{k=2}^{n-2}b_kx_{k+1}) = $$
$$b_1x_{2} + b_{n-1}\sum\limits_{k=2}^{n-1}a_kx_{k} + b_n(b_1x_1 -b_n x_n).$$

On the other hand
$$[x_1,[x_n,x_1]] = [x_1, b_1x_1 -b_n x_n]=  b_1x_2 -b_n \sum\limits_{k=2}^{n-1}b_kx_{k}+ b_n(b_1x_1 -b_n x_n).$$

Comparing the coefficients at the basic elements, we deduce $$b_{n-1}a_k = -b_nb_k, \ 2 \leq k \leq n-1.$$

Consider the Leibniz identity $[x_i,[x_n,x_1]]$ with $2 \leq i \leq n-2,$
$$[x_i,[x_n,x_1]] = [[x_i, x_n],x_1] - [[x_i,
x_1],x_n]= [-ib_1x_{i} + \sum\limits_{k=2}^{n-i}b_kx_{k+i-1},x_1]  - [x_{i+1}, x_n]=$$
$$-ib_1x_{i+1} + \sum\limits_{k=2}^{n-i-1}b_kx_{k+i} +b_{n-i}\sum\limits_{k=2}^{n-1}a_kx_{k} -
(-(i+1)b_1x_{i+1} + \sum\limits_{k=2}^{n-i-1}b_kx_{k+i}) = $$
$$b_1x_{i+1} + b_{n-i}\sum\limits_{k=2}^{n-1}a_kx_{k}.$$

On the other hand, we have
$$[x_i,[x_n,x_1]] = [x_i, b_1x_1 -b_n x_n]=  b_1x_{i+1} -b_n (-ib_1x_{i} + \sum\limits_{k=2}^{n-i}b_kx_{k+i-1}).$$

Comparing the coefficients at the basic elements, we derive
$$\begin{cases}b_{n-i}a_k = 0, & 2 \leq i \leq n-2, \ 2 \leq k \leq i-1,\\
b_{n-i}a_i = ib_nb_1, & 2 \leq i \leq n-2, \\
b_{n-i}a_k = -b_nb_{k-i+1}, & 2 \leq i \leq n-2, \  i+1 \leq k \leq n-1.\end{cases}
$$

Similarly, from equalities
$$[x_{n-1},[x_n,x_1]] = [[x_{n-1}, x_n],x_1] - [[x_{n-1},
x_1],x_n]= -(n-1)b_1[x_{n-1},x_1]  -
\big[\sum\limits_{k=2}^{n-1}a_kx_{k}, x_n\big]=$$
$$-(n-1)b_1\sum\limits_{k=2}^{n-1}a_kx_{k} - \sum\limits_{k=2}^{n-1}a_k (-kb_1x_{k}
+
\sum\limits_{s=2}^{n-k}b_sx_{s+k-1})=b_1\sum\limits_{k=2}^{n-2}(-n+1+k)a_kb_1x_{k}
- \big(\sum\limits_{k=2}^{n-2}a_k b_{n-k}\big)x_{n-1}=$$
$$b_1\sum\limits_{k=2}^{n-2}(-n+1+k)a_kb_1x_{k} - \sum\limits_{k=2}^{n-2}ib_1b_n x_{n-1}=
b_1\sum\limits_{k=2}^{n-2}(-n+1+k)a_kb_1x_{k} - \frac {n(n-3)b_1b_n} 2x_{n-1}$$
and equalities
$$[x_{n-1},[x_n,x_1]] = [x_{n-1}, b_1x_1 -b_n x_n]=  b_1\sum\limits_{k=2}^{n-1}a_kx_k + (n-1)b_nb_1x_{n-1},$$
we obtain
$$\begin{cases}(n-k)a_kb_1 =0, & 2 \leq k \leq n-2,\\
(a_{n-1}+\frac {(n+1)(n-2)b_n} 2)b_1=0.&\end{cases}$$

Consider
$$[x_{n},[x_n,x_1]] = [[x_{n}, x_n],x_1] - [[x_{n},
x_1],x_n]= b_{n+1}[x_{n-1},x_1]  - [b_1x_1-b_nx_n, x_n]=$$
$$b_{n+1}\sum\limits_{k=2}^{n-1}a_kx_{k} - b_1 (-b_1x_{1} + \sum\limits_{k=2}^{n}b_kx_{k}) +b_nb_{n+1}x_{n-1}=$$$$
b_1^2x_{1} + \sum\limits_{k=2}^{n-2}(b_{n+1}a_k -b_1b_k)x_{k} + (b_{n+1}a_{n-1} -b_1b_{n-1}+ b_nb_{n+1})x_{n-1}- b_1b_nx_{n}.$$

On the other hand, we have
$$[x_{n},[x_n,x_1]] = [x_{n}, b_1x_1 -b_n x_n]= b_1^2x_{1} - b_nb_{n+1}x_{n-1} - b_1b_nx_{n}.$$

Comparing the appropriate coefficients, we conclude
$$\begin{cases}b_1b_k = b_{n+1}a_k , & 2 \leq k \leq n-2,\\
b_1b_{n-1} = b_{n+1}(a_{n-1} + 2b_n).&\end{cases}$$

Let us summarize the above restrictions:
$$\begin{cases}
b_{n-1}a_k = -b_nb_k, & 2 \leq k \leq n-1,\\
b_{n-i}a_k = 0, & 2 \leq i \leq n-2, \ 2 \leq k \leq i-1,\\
b_{n-i}a_i = ib_nb_1, & 2 \leq i \leq n-2, \\
b_{n-i}a_k = -b_nb_{k-i+1}, & 2 \leq i \leq n-2, \ i+1 \leq k \leq n-1,\\
(n-k)a_kb_1 =0, & 2 \leq k \leq n-2,\\
(a_{n-1}+\frac {(n+1)(n-2)b_n} 2)b_1=0,\\
b_1b_k = b_{n+1}a_k , & 2 \leq k \leq n-2,\\
b_1b_{n-1} = b_{n+1}(a_{n-1} + 2b_n).
\end{cases}
$$
We need to consider the following distinguish cases.

\textbf{Case 1.} Let $b_1 \neq 0.$ Then $a_k=0,$ $2 \leq k \leq n-1$ and $b_k=0,$ $2 \leq k \leq n.$
In this case  case we have the table of multiplication
$$\ \left\{\begin{array}{ll}
[x_i,x_1]=x_{i+1}, & \  1\leq i \leq {n-2},\\[1mm]
[x_{n},x_1]=b_1x_1,\\[1mm]
[x_i,x_n]=-ib_1x_{i}, & 1\leq i \leq {n-1},\\[1mm]
[x_n,x_n]=b_{n+1}x_{n-1}, & 1\leq i \leq {n-1}.
\end{array} \right.$$
Taking the change $x_n' = \frac 1 {b_1} x_n + \frac {b_{n-1}} {b_1(n-1)} x_{n-1}$ we obtain the algebra $\nu_2$.

\textbf{Case 2.} Let $b_1 = 0.$ Then the above restrictions are reduced to the following one:
\begin{equation}\label{E:3.5}\begin{cases}
b_{n-1}a_k = -b_nb_k, & 2 \leq k \leq n-1,\\
b_{n-i}a_k = 0, & 2 \leq i \leq n-2, \ 2 \leq k \leq i,\\
b_{n-i}a_k = -b_nb_{k-i+1}, & 2 \leq i \leq n-2, \  i+1 \leq k \leq n-1,\\
b_{n+1}a_k =0, & 2 \leq k \leq n-2,\\
b_{n+1}(a_{n-1} + 2b_n)=0.
\end{cases}
\end{equation}

\begin{itemize}
\item If there exists some $b_i\neq 0,$ where $2 \leq i \leq n-1,$ then from (\ref{E:3.5}) we obtain
$$a_i=0,\ 2 \leq i \leq n-1, \quad a_{n-1}=-b_n, \quad b_nb_{n+1}=0.$$

\begin{itemize}
\item Let us suppose $b_n=0.$ Then we have the multiplication
$$\left\{\begin{array}{ll}
[x_i,x_1]=x_{i+1}, & \  1\leq i \leq {n-2},\\[1mm]
[x_i,x_n]=\sum\limits_{k=2}^{n-i}b_kx_{k+i-1}, & 1\leq i \leq {n-2},\\[1mm]
[x_n,x_n]=b_{n+1}x_{n-1}. & \\[1mm]
\end{array} \right.$$

This family of algebras represents families of filiform Leibniz algebras $F_1$ and $F_2$. Namely, if $b_2\neq 0,$ then we obtain the family $F_1$
and if $b_2=0,$ then we get the family $F_2.$

\item Let us assume now that $b_{n} \neq 0.$ Then $b_{n+1} =0$ and by scaling the basis elements, one can assume $b_n=1.$
Thus, we obtain the algebra $\nu_3(b_2, b_3, \dots, b_{n-1}).$
\end{itemize}

\item If $b_i = 0$ for $2 \leq i \leq n-1.$
\begin{itemize}
\item Let $b_{n+1}\neq 0.$ Then $a_i=0,\ 2 \leq i \leq n-1$ and $a_{n-1}=-2b_n.$ In the case of of $b_n =0,$ we have a filiform Leibniz algebra of the family $F_2$ and
in the case of $b_n \neq 0,$ by scaling of appropriate basis elements,
we can suppose $b_n=b_{n+1}=1.$ Thus, we get the algebra $\nu_4.$

\item Let  $b_{n+1}= 0.$ Since $x_n \notin Ann_r(\nu)$ one can conclude $b_n \neq 0.$
By scaling the basis elements, we can assume $b_n=1$ and the algebra $\nu_5(a_2, a_3, \dots, a_{n-1})$ is obtained.
\end{itemize}
\end{itemize}
\end{proof}

Below we give some remarks concerning algebras $\nu_1 - \nu_5$.
\begin{rem}

\

1)  Since the following single-generated Leibniz algebra:
$$\left\{\begin{array}{ll}
[x_i,x_1]=x_{i+1}, & 1\leq i \leq {n-2},\\[1mm]
[x_{n-1},x_1]=x_n + \sum\limits_{k=2}^{n-1}a_kx_{k}, &\\[1mm]
[x_{n},x_1]=x_n.
\end{array} \right.$$
degenerates to the algebra $\nu_1(a_2, a_3, \dots, a_{n-1})$ via
family of transformations $g_t$, which are given as follows
$$g_t(x_n) = tx_n, \quad g_t(x_i) = x_i, \ 1 \leq i \leq n-1,$$
we conclude $\nu_1(a_2, \dots, a_{n-1}) \in X.$

2) Note that the algebra $\nu_2$ is the unique (up to isomorphism) solvable Leibniz algebra with null-filiform nilradical \cite{Cas1}. Due to work \cite{Anc} the algebra $\nu_2$ is rigid.

3) The algebra $\nu_3$ is a solvable Leibniz algebra with nilradical $N= <x_2, x_3, \dots, x_{n}>$, which has the table of multiplication:
$$[x_i,x_n]=\sum\limits_{k=2}^{n-i}b_kx_{k+i-1}, \quad  2\leq i \leq {n-2}.$$
In particular, if $b_2\neq 0,$ then $N$ is a filiform algebra.

4) The algebra $\nu_4$ is a solvable Leibniz algebra with
nilradical $N=<x_2, x_3, \dots, x_{n}>,$ which is isomorphic to
the direct sum of two-dimensional Leibniz algebra
$[x_n,x_n]=x_{n-1}$ and $\mathbb{C}^{n-3}$. It should be noted
that algebra $N$ is one of the fourth algebras of level one
\cite{KOLevelone}.

5) The algebra $\nu_5$ is a solvable Leibniz algebra with abelian nilradical $<x_2, x_3, \dots, x_{n}>$.
\end{rem}

\subsection{On the description of Leibniz infinitesimal deformations of the algebra $F_n^3$.}

\

In this subsection give some additional information on Leibniz
infinitesimal deformations of $F_n^3(0).$

Recall, a Lie infinitesimal deformation $\varphi$ is defined as
bilinear map which satisfies the equality \eqref{E.Z2} and
skew-symmetric condition $\varphi(x,y) = - \varphi(y,x)$
\cite{GozKhak}.

Thanks to the works \cite{Ver} and \cite{Khak}, where the
infinitesimal deformations of the algebra $F_n^3(0)$ in the
varieties of nilpotent and all Lie algebras are described,
respectively, it is sufficient to study the infinitesimal Leibniz
deformations which do not satisfy skew-symmetric condition.

Denote by $ZL^2(F_n^3(0),F_n^3(0))$ and $Z^2(F_n^3(0), F_n^3(0))$ the space of
Leibniz and Lie infinitesimal deformations, respectively.

\begin{thm}\cite{Khak} \label{thm3.19} If $Z^2(F_n^3(0), F_n^3(0))$ is a vector space of
the infinitesimal deformations of $F_n^3(0)$ on the $n$-dimensional
Lie algebra laws $Lie_n,$ then
$$dim Z^2(F_n^3(0), F_n^3(0)) = \begin{cases}8,&  n=3,\\[1mm] 15,&  n=4,\\[1mm]
\frac {(n-1)(3n-5)} 8 + n^2-n-1,& n  \ \mbox{is odd and} \ n \geq 5,\\[1mm]
\frac {n(3n-10)} 8 + n^2-n + [\frac {n} 4],& n \ \mbox{is even and} \ n \geq 5.\\[1mm] \end{cases}$$
\end{thm}

In the next proposition we present a general form of non-Lie Leibniz infinitesimal deformations of the algebra
$F_n^3(0).$

\begin{prop} If $ \varphi \in ZL^2( F_n^3(0),F_n^3(0)),$ then
$$ \varphi(x_1, x_1) = \alpha x_n,  \quad   \varphi(x_1, x_2) +
\varphi(x_2, x_1) = \beta x_n, \quad \varphi(x_2, x_2) = \gamma
x_n,$$
$$\varphi(x_i,x_j)=-\varphi(x_j,x_i)\quad \mbox{for other i and j}.$$
\end{prop}

\begin{proof} Let  $ \varphi \in
ZL^2(F_n^3(0),F_n^3(0)).$ We shall consider the equation
\eqref{E.Z2} for elements $x_i, x_j, x_k.$

For $j=k=1$ we have $[x_i,\varphi(x_1,x_1)]=0,$ consequently one
can suppose
$$\varphi(x_1,x_1)=\alpha x_n.$$

Similarly, for $j=k=2$ we obtain  $[x_i,\varphi(x_2,x_2)]=0,$
which deduces $$\varphi(x_2,x_2)=\gamma x_n.$$

For $i=j=1$ and any $k$ with $2\leq k\leq n-1$ we conclude
$\varphi(x_1,[x_1,x_k])+\varphi([x_1,x_k],x_1) = 0,$ which implies
\begin{equation}\label{E:3.8}\varphi(x_1,x_{k+1})=-\varphi(x_{k+1},x_1), \ 2\leq k\leq
n-1. \end{equation}

For $i=k=1, j=2$ we derive $[x_1,\varphi(x_2,x_1)]-[\varphi(x_1,x_2),x_1]=0$ from which we can assume
\begin{equation}\label{E:3.9}\varphi(x_1,x_2)+\varphi(x_2,x_1)=\beta_1x_1+\beta_n x_n.\end{equation}

Considering the condition $(d^2\varphi )(x_i, x_1, x_k) = 0$ with
$2\leq i\leq n,$ $2\leq k\leq n-1$ we have
\begin{equation}\label{E:4.1}[x_i,\varphi(x_1,x_k)]-[\varphi(x_i,x_1),x_k]+[\varphi(x_i,x_k),x_1]-\varphi(x_i,x_{k+1})-\varphi([x_i,
x_1],x_k)=0. \end{equation}

Analogously, from $(d^2\varphi )(x_1, x_j, x_k) = 0$ with $2\leq
j\leq n,$ $2\leq k\leq n-1$ we obtain
\begin{equation}\label{E:4.2}[x_1,\varphi(x_j,x_k)]-[\varphi(x_1,x_j),x_k]+[\varphi(x_1,x_k),x_j]-
\varphi([x_1,x_j],x_k)-\varphi(x_{k+1},x_j)=0.\end{equation}

Combining the equalities \eqref{E:4.1} and \eqref{E:4.2} we derive
$$\varphi(x_i,x_{k+1})+\varphi(x_{k+1},x_i)=[x_k,\varphi(x_i,x_1) +\varphi(x_1,
x_i)],\quad 2\leq i\leq n, \ 2\leq k\leq n-1.$$

Thanks to above equality and \eqref{E:3.8}--\eqref{E:3.9} we
obtain
$$\varphi(x_2,x_{k+1})+\varphi(x_{k+1},x_2)=\beta_1x_{k+1}, \quad \varphi(x_i,x_{k+1})=-\varphi(x_{k+1},x_i)$$ with $3\leq i\leq
n,2\leq k\leq n-1.$

The restriction $\beta_1=0$ follows from $(d^2\varphi )(x_2, x_2,
x_1) = 0$.
\end{proof}

\begin{thm} The cochains  $\psi_1,\psi_2,\psi_3$
defined as
$$\psi_1(x_1,x_1)=x_n,\quad \psi_2(x_1,x_2)=x_n,\quad \psi_3(x_2,x_2)=x_n$$ complement the
subspace $Z^2(F_n^3(0),F_n^3(0))$ to the space $ZL^2(F_n^3(0),F_n^3(0))$.
\end{thm}

\begin{cor}
$$dim ZL^2(F_n^3(0), F_n^3(0)) = \begin{cases} 11,& n=3,\\[1mm] 18,&   n=4,\\[1mm]
\frac {(n-1)(3n-5)} 8 + n^2-n+2,& n \ \mbox{is odd and} \ n \geq 5,\\[1mm]
\frac {n(3n-10)} 8 + n^2-n +3+ [\frac {n} 4],&  n \ \mbox{is even and} \ n \geq 5.\\[1mm] \end{cases}$$
\end{cor}

It is known that any derivation of algebra $F_n^3(0)$ with $n=3$ and $n\geq 4$ has the following matrix form \cite{GozKhak}:
$$\begin{pmatrix}
  a_1 & a_2 & a_3 \\
 b_1 & b_2 & b_3 \\
 0 & 0 & a_1+b_2 \\
\end{pmatrix}, \quad \begin{pmatrix}
\alpha_1& \alpha_2&\alpha_3&\alpha_4&\dots&\alpha_{n-1}&\alpha_n\\
0& \beta_2&\beta_3&\beta_4&\dots&\beta_{n-1}& \beta_n\\
0& 0&\alpha_1+\beta_2&\beta_3&\dots&\beta_{n-2}& \beta_{n-1}\\
0& 0&0&2\alpha_1+\beta_2&\dots&\beta_{n-3}&\beta_{n-2}\\
\vdots&\vdots&\vdots&\vdots&\dots&\vdots&\vdots\\
0&0&0&0&\dots&0&(n-2)\alpha_1+\beta_2
\end{pmatrix}.$$

Therefore,
$$dim Der(F_n^3(0)) = \begin{cases} 6,& n=3,\\[1mm] 2n-1,&
n  \geq 4. \end{cases} \quad  dim BL^2(F_n^3(0), F_n^3(0)) = \begin{cases} 3,&  n=3,\\[1mm] (n-1)^2,& n\geq 4.\\[1mm]
\end{cases}$$

Consequence of Theorem \ref{thm3.19} is the following result.

\begin{cor}
$$dim HL^2(F_n^3(0), F_n^3(0)) = \begin{cases} 9, & n=3,\\[1mm] 10, & n=4,\\[1mm]
\frac {(n-1)(3n-5)} 8 + n+2,&  n \ \mbox{is odd and} \ n \geq 5,\\[1mm]
\frac {n(3n-10)} 8 + n+3+ [\frac {n} 4],&  n \ \mbox{is even and} \ n \geq 5.\\[1mm] \end{cases}$$
\end{cor}

\end{document}